\documentclass[12pt]{article}
\usepackage{amsmath}
\usepackage{amssymb}
\newlength{\mywidth}

\usepackage{amsmath,amssymb,textcomp,graphicx,amsthm} 
\usepackage{amssymb,verbatim,color}
\usepackage{esint}
\usepackage{tikz}
\usepackage[colorlinks=true,urlcolor=blue, citecolor=red,linkcolor=blue,linktocpage,pdfpagelabels, bookmarksnumbered,bookmarksopen]{hyperref}

\newtheorem{thm}{Theorem}
\newtheorem{lemma}[thm]{Lemma}
\newtheorem{cor}[thm]{Corollary}
\newtheorem{prop}[thm]{Proposition}

\theoremstyle{definition}
\newtheorem{rem}[thm]{Remark}
\newtheorem{ex}{Example}

\DeclareMathOperator{\ddiv}{div}

\DeclareMathOperator*{\osc}{osc}

\DeclareMathOperator{\dist}{dist}
\newcommand{\ba}{\boldsymbol{\alpha}}
\newcommand{\bbeta}{\boldsymbol{\beta}}
\newcommand{\bg}{\boldsymbol{\gamma}}
\newcommand{\bs}{\boldsymbol{\sigma}}
\newcommand{\bo}{\boldsymbol{\omega}}
\newcommand{\bm}{\boldsymbol{\mu}}
\newcommand{\bl}{\boldsymbol{\lambda}}
\newcommand{\R}{\mathbb{R}}
\newcommand{\e}{\varepsilon}
\def\Xint#1{\mathchoice
   {\XXint\displaystyle\textstyle{#1}}%
   {\XXint\textstyle\scriptstyle{#1}}%
   {\XXint\scriptstyle\scriptscriptstyle{#1}}%
   {\XXint\scriptscriptstyle\scriptscriptstyle{#1}}%
   \!\iint}
\def\XXint#1#2#3{{\setbox0=\hbox{$#1{#2#3}{\iint}$}
     \vcenter{\hbox{$#2#3$}}\kern-.51\wd0}}

\def\longminus{\raisebox{-2ex}{\rotatebox[origin=c]{20}{${-}\mkern-3.5mu{-}$}}}
  
\def\fiint{\Xint\longminus}
\begin{document}
\settowidth{\mywidth}{ab}
\title{The Gradient Flow of Infinity-Harmonic Potentials}
\author{Erik Lindgren, Peter Lindqvist}

\date{\today}
\maketitle

\medskip
    {\small \textsc{Abstract:} \textsf{We study the streamlines of $\infty$-harmonic functions in planar convex rings. We include convex polygons.    The points where streamlines can meet are characterized: they lie on certain curves.  
        The gradient has constant norm along streamlines outside the set of meeting points,  the \emph{infinity-ridge}.
}}
      
\bigskip

\tableofcontents 

\bigskip 
\bigskip 

{\small \textsf{AMS Classification 2010}: 49N60, 35J15, 35J60, 35J65, 35J70.} 

{\small \textsf{Keywords}: Infinity-Laplace Equation, streamlines, convex rings, infinity-potential function}
\section{Introduction} 
The $\infty$-Laplace Equation
$$
 \Delta_{\infty}u\,\equiv\,\sum_{i,j}\frac{\partial u}{\partial x_i} \frac{\partial u}{\partial x_j}\frac{\partial^2 u}{\partial x_i \partial x_j}\,=\,0
 $$
was introduced by G. Aronsson in 1967 (cf. \cite{A1}) to produce optimal Lip\-schitz extensions of boundary values. It has been extensively studied. Some of the highlights are
\begin{itemize}
\item Viscosity solutions for $\Delta_{\infty}$, \cite{BDM}
\item Uniqueness, \cite{J}
\item Differentiability, \cite{S}, \cite{ESa} and \cite{ES}
\item Tug-of-War (connection with stochastic game theory), \cite{PSW}
\end{itemize}
We are interested in the two-dimensional equation 
$$
  \Bigl(\frac{\partial u}{\partial x_1}\Bigr)^{\!2}\frac{\partial ^2u}{\partial x_1^2}\,+\,2\, \frac{\partial u}{\partial x_1}\frac{\partial u}{\partial x_2}\frac{\partial ^2u}{\partial x_1 \partial x_2}\,+\, \Bigl(\frac{\partial u}{\partial x_2}\Bigr)^{\!2}\frac{\partial ^2u}{\partial x_2^2}\,=\,0
  $$
  in so-called convex ring domains $G = \Omega\setminus K$. Here $\Omega$ is a bounded convex domain in $\R^2$ and $K\Subset \Omega$ is a closed convex set.
  We continue our investigation in \cite{LL} of the $\infty$-potential $u_\infty$, which is the unique solution in $C(\overline G)$ of the boundary value problem 
$$
    \begin{cases} \Delta_{\infty}u\,=\,0\qquad\text{in}\qquad G\\ \phantom{ \Delta_{\infty}}
      u\,=\,0\qquad\text{on}\qquad \partial \Omega\\\phantom {\Delta_{\infty}}
      u\,=\,1 \qquad\text{on}\qquad \partial K.
    \end{cases}
  $$
In \cite{LL} we proved that the \emph{ascending} streamlines, the solutions $\ba= (\alpha_1,\alpha_2)$ of 
  $$
  \frac{d\ba(t)}{dt} = +\nabla u_\infty (\ba(t)), \quad 0\leq t < T_{\ba}
  $$
  with given initial point $\ba(0)\in \overline \Omega\setminus K$, are unique and terminate at $\partial K$. (The descending ones are not!) Streamlines may meet and then continue along a common arc. Uniqueness prevents crossing streamlines.
  
  Along a streamline one would expect that the speed $|\nabla u_\infty (\ba)|$ is constant. Indeed, 
  $$  
  \frac{d}{dt} |\nabla u_\infty (\ba(t))|^2 = 2\,\Delta_\infty u_\infty (\ba (t)) = 0, 
  $$
  but the calculation requires second derivatives. The main difficulty is the lack of second derivatives. Although, the second derivatives are known to exist  almost everywhere with respect to the Lebesgue area, see \cite{KZZ} for this new result, this is of little use since the area of a streamline is zero. In \cite{LL} it was shown that the above calculation fails: for most streamlines the speed is not constant the whole way up to $\partial K$. (We shall see that the speed is constant from the initial point till the streamline meets another streamline.) 
  
  We use the approximation with the (unique) solution of the $p$-Laplace equation
  $$
\Delta_p u = \ddiv (|\nabla u|^{p-2}\nabla u) = 0, \qquad p > 2.
 $$
 in $G$ with the same boundary values as $u_{\infty}$.

We shall use several facts about these $p$-harmonic functions due to J. Lewis, cf.  \cite{L}. It is decisive that the \emph{level curves} $\{u_p(x)=c\}$ are convex and that $\Delta u_p \leq 0$. See Section \ref{sec:prel} for more details.
  
  We also need the facts that (i) $\nabla u_p\to \nabla u_\infty$ in $L^2_\text{loc}$ and (ii) the family $\{|\nabla u_p|\}$ is locally equicontinuous. (Notice that we wrote $|\nabla u_p|$, not $\nabla u_p$.) We extract a proof of this from the recent pathbreaking work by H. Koch, Y. R-Y. Zhang and Y. Zhou in \cite{KZZ}, complementing  their results by applying a simple device, due to Lebesgue in \cite{Leb}, to the norm $|\nabla u_p|$ of the quasiregular mapping 
  $$
  \frac{\partial u_p}{\partial x_1}-  i\frac{\partial u_p}{\partial x_2}, \quad i^2=-1.
  $$The quasiregularity was obtained by B. Bojarski and T. Iwaniec in \cite{BI}.

We prove the following basic result in Section \ref{sec:eqcont}.
  
    \begin{thm}[Non-decreasing speed]  \label{thm:speed} Let $\ba_\infty = \ba_\infty(t)$, $0\leq t\leq T$, be a streamline of $u_\infty$, i.e., 
  $$
  \frac{d\ba_\infty(t)}{dt} = \nabla u_\infty (\ba_\infty(t)), \quad 0\leq t< T, 
  $$
  and $\ba_\infty(0)\in \partial \Omega$, $\ba_\infty(T)\in \partial K$. Then the function $u_\infty(\ba_\infty(t))$ is  convex when  $0\leq t\leq T$. In particular, the speed $|\nabla u_\infty (\ba_\infty(t))|$, is a non-decreasing function of $t$.
  \end{thm}

   Combining this with a result in the opposite direction (cf. Lemma 12 in \cite{LL}), we can control the meeting points so that these lie on a few specific streamlines, here called attracting streamlines.
    
  \paragraph{Polygons.} To avoid a complicated description, we begin with a convex polygon as $\Omega$ with $N$ vertices $P_1,P_2,\ldots, P_N$ (set $P_{N+1}=P_1$ for convenience). With $P_k = \bg_k(0)$ as initial point there is a unique streamline
  $$
  \bg_k = \bg_k(t), \quad  0\leq t \leq T_k,
  $$
  with terminal point $\bg_k(T_k)$ on $\partial K$. The $$\text{\emph{attracting streamlines} are}\qquad \bg_1, \bg_2, \ldots, \bg_N.$$ 
  Occasionally, some of them meet and then share a common arc up to $\partial K$. The collection of all the points on the attracting streamlines is called the $\infty$-\emph{ridge} and is denoted by $\Gamma$, i.e., 
  $$
  \Gamma = \bigcup_{k=1}^N \{\bg_k(t): \,\, 0\leq t\leq T_k\}.
  $$
It seems to play a similar role for the $\infty$-Laplace Equation as the (ordinary) ridge does for the Eikonal Equation. 

Before meeting any other streamline, a streamline $\ba$ either meets an attracting streamline or hits the upper boundary $\partial K$. We formulate this as a theorem,  proved in Section \ref{sec:poly}.
\begin{thm}\label{thm:mainpoly} The speed  $|\nabla u_\infty(\ba(t))|$ is constant along the streamline $\ba$ from the initial point on $\partial \Omega$ until it meets one of the attracting streamlines $\bg_k$, after which the speed is non-decreasing. It cannot meet any other streamline before it meets an attracting one. 
\end{thm}

Thus there are no meeting points in $G\setminus \Gamma$, i.e., they all lie on the attracting streamlines $\bg_1, \bg_2, \ldots, \bg_N$. In other words, there is no branching outside the $\infty$-ridge $\Gamma$.   

\paragraph{General Domains.}  The polygon has a piecewise smooth boundary and at the vertices $|\nabla u_\infty(P_k)|= 0$. Thus the attracting streamlines start at the points of minimal speed. Similar results hold when $\Omega$ is no longer a polygon, but now we have to assume that the following holds:

\bigskip 
\noindent\textbf{Assumptions:}
\begin{enumerate}
\item\emph{$\nabla u_\infty$ is continuous in $\overline \Omega\setminus K$, in particular along $\partial \Omega$}.\footnote{For example, if $\partial \Omega$ is piecewise $C^2$, then the gradient is continuous in $\overline \Omega\setminus K$, see Section \ref{sec:prel}.}
\item \emph{On $\partial \Omega$, the continuous function $|\nabla u_\infty|$ has a finite number of local minimum points, say $P_1,P_2,\ldots, P_N$, and a finite number of local maximum points.}
\end{enumerate}
 
  Again, the streamlines with the initial points $P_k$ are called \emph{attracting streamlines}:
   $$
  \bg_k = \bg_k(t), \quad  0\leq t\leq T_k ; \quad \bg_k(0)=P_k.
  $$
  The $\infty$-\emph{ridge} is again
  $$
  \Gamma = \bigcup_{k=1}^N \{\bg_k(t): \,\, 0\leq t\leq T_k\}.
  $$
  Theorem \ref{thm:mainpoly} holds also in this setting. As a consequence, streamlines cannot meet, except on $\Gamma$. The theorem below is proved in Section \ref{sec:general}.
  \begin{thm}\label{thm:general} The speed  $|\nabla u_\infty(\ba(t))|$ is constant along a streamline $\ba$ from the initial point on $\partial \Omega$ until it meets one of the attracting streamlines $\bg_k$. It cannot meet any other streamline before it meets an attracting one.
  \end{thm}

  The situation when $|\nabla u_{\infty}|$ is constant on some arc on $\partial \Omega$ can happen even for a rectangle, but does not cause extra complications. 

  \begin{prop} If the speed $|\nabla u_{\infty}|$ is constant along a boundary arc $\overline{ab}$, then the streamlines with initial points on the arc are non-intersecting segments of straight lines. They meet no other streamlines in $G$, except possibly when the initial point is $a$ or $b$.
  \end{prop}

  This follows from Lemma \ref{lem:constant} and Lemma \ref{lem:trilem}. It  allows us to relax assumption 2 to include boundary arcs with constant local maximum speed:
  \begin{itemize}
  \item[2*.]
  \emph{The local maxima and minima of  $|\nabla u_{\infty}|$   on $\partial \Omega$ are attained along at most finitely many closed subarcs, which may degenerate  to points.}
  \end{itemize}
  
  The definition of the attracting streamlines must be amended if the speed attains a local minimum along a boundary arc $\overline{ab}$: it contributes with \emph{two} attracting streamlines, namely the ones with initial points at $a$ and $b$.

 \begin{rem}  The behavior of the streamlines suggests that the $\infty$-potential is smooth outside the $\infty$-ridge $\Gamma$.
  \end{rem}
  
  \paragraph{Examples.}
  We mention some examples.
\begin{figure}[h!]
\begin{center}
\includegraphics[scale=0.05]{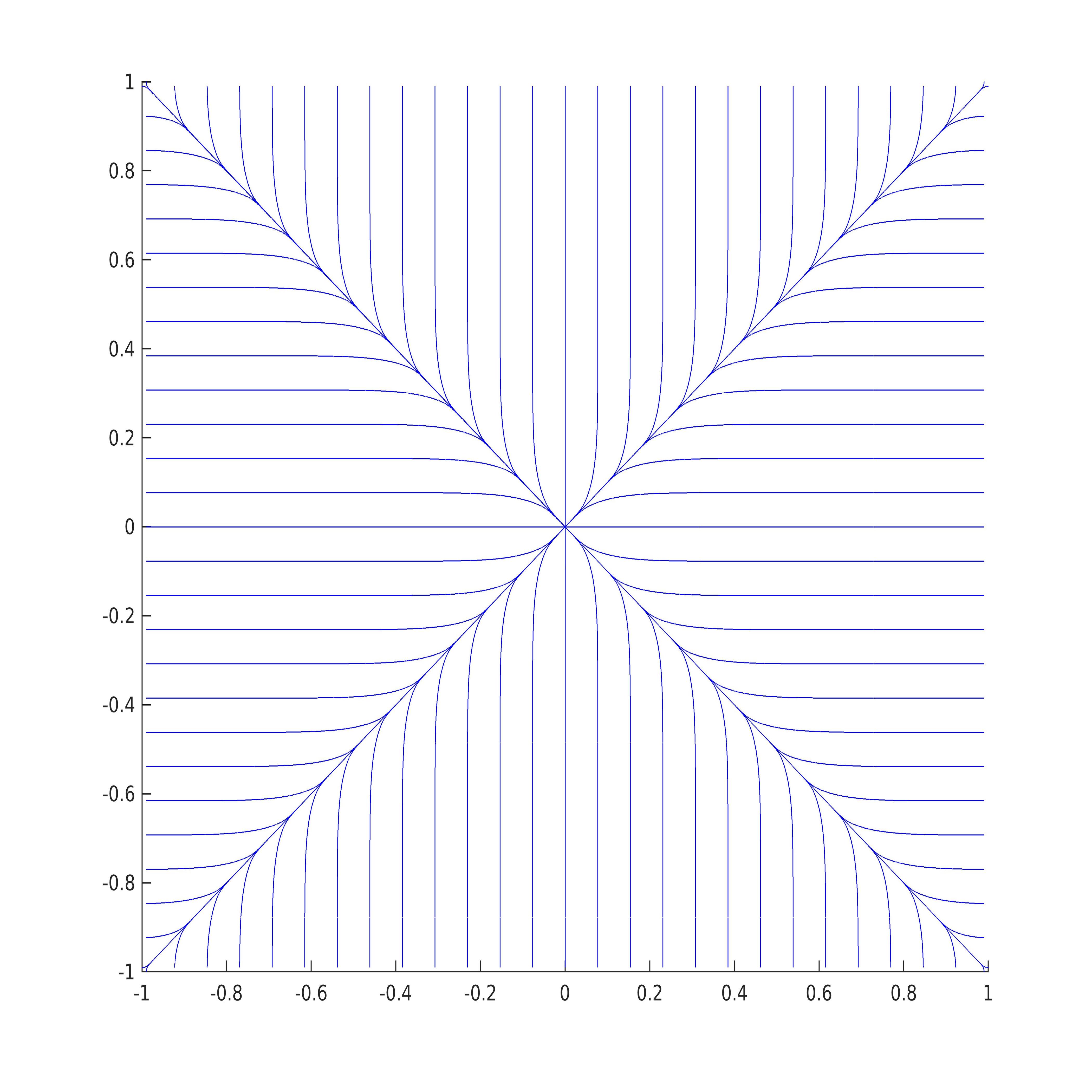}
\caption{The streamlines of $u_\infty$ when $\Omega$ is the square in Example \ref{ex:square}.}
\label{fig:square}
\end{center}
\end{figure}

\begin{ex}\label{ex:square}
Let $\Omega$ be the square
$$
-1<x_1<1, \quad -1<x_2<1, \quad
$$
and $K$ the origin. The attracting streamlines are the four half-diagonals, constituting the $\infty$-ridge
$$
\Gamma = \{(x_1,x_2): \quad x_1=\pm x_2,  |x_1|\leq 1,  |x_2|\leq 1\}.
$$
All streamlines meet at a diagonal, except the four segments along the coordinate axes. See Figure \ref{fig:square}.
\end{ex}

\begin{ex}\label{ex:modsq} Let $K$ be the origin and $\Omega$ the square in Example \ref{ex:square} which is truncated in the following symmetric way: in the south west corner we have removed the triangle with corners $(-1,-1), (-1+\delta,-1)$ and $(-1,-1+\delta)$, for some small $\delta$. See Figure \ref{fig:modsq}. We only describe the behavior in the south west quarter of $\Omega$.
	
The attracting streamlines are those starting in $(-1+\delta,-1)$ and $(-1,-1+\delta)$ (in blue). The only streamlines that do not meet any other before reaching origin, are the medians (in red). Any other streamline will meet one of the attracting streamlines. The streamline starting in the middle of $(-1+\delta,-1)$ and $(-1,-1+\delta)$ (in red) will be a straight line to the origin and will be joined by the attracting streamlines from both sides before terminating at the origin.

\end{ex}

 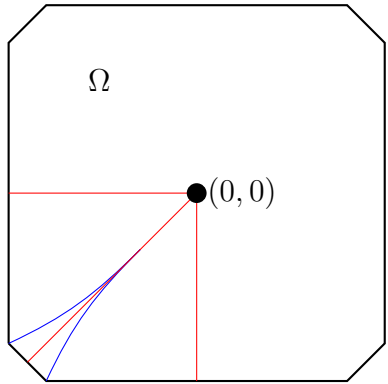
\begin{figure}[h!]
  \begin{center}
 
	\begin{tikzpicture}[domain=-3:3,scale=2.5]

\draw[thick] (-1,-0.8) -- (-0.8,-1) -- (0.8,-1) -- (1,-0.8) -- (1,0.8) -- (0.8,1) -- (-0.8,1) -- (-1,0.8) -- (-1,-0.8);
\draw[color=red] (-1,0) -- (0,0);
\draw[color=red] (0,-1) -- (0,0);
\draw[color=blue] (-0.8,-1)  to[out=65, in =-135] (-0.3,-0.3);
\draw[color=blue] (-1,-0.8) to[out=25, in =-135] (-0.3,-0.3);
\draw[color=red] (-0.9,-0.9) -- (0,0);

 \fill (-0.4,0.6) node[left] {$\Omega$};
\fill (0,0) circle[radius=1.5pt] node[right] {$(0,0)$};

	\end{tikzpicture}
	
\end{center}
\caption{The truncated square in Example \ref{ex:modsq} and some possible streamlines.}
 \label{fig:modsq}  
	\end{figure}

  \section{Preliminaries} \label{sec:prel} $\Omega$ is a bounded convex domain in $\R^2$  and $K\Subset \Omega$ is a compact and convex set, which may reduce to a point. We study the equation in the convex ring $G=\Omega\setminus K$. We assume the following \emph{normalization}: 
  $$\boxed{\dist(\partial \Omega,K)=1.}$$ 
  The boundary value problem 
$$
 \begin{cases} \Delta_{\infty}u\,=\,0\qquad\text{in}\qquad G,\\ \phantom{ \Delta_{\infty}}
      u\,=\,0\qquad\text{on}\qquad \partial \Omega,\\\phantom {\Delta_{\infty}}
      u\,=\,1 \qquad\text{on}\qquad \partial K,
    \end{cases}
    $$
has a unique solution $u_\infty \in C(\overline G)$ in general. By \cite{ESa}, $\nabla u_\infty$ is locally H\"older continuous in $G$. We will assume that also $\nabla u_\infty\in  C(\overline\Omega\setminus K)$. This is fulfilled if for instance $\partial \Omega$ has a piecewise $C^2$ regular boundary. See Lemma 2 and Theorem 2 in \cite{HL}, Theorem 7.1 in \cite{MPS} and Theorem 1 in \cite{WY}.

In \cite{LL} it was established that, for a given initial point $\xi_0\in \partial\Omega$, the gradient flow
$$
\begin{cases}
\displaystyle\frac{d \ba (t)}{d t}& =\, +\nabla u_\infty(\ba (t)),\quad    0\leq t< T, \\
\ba(0)&=\,\xi_0,
\end{cases}
$$
has a unique solution $\ba=\ba(t)$, which terminates at some point $\ba(T)$ on $\partial K$. (Some caution is required if $|\nabla u_\infty(\xi_0)| = 0$.) We say that $\ba$ is a \emph{streamline}. Although unique, two streamlines may meet, join, and continue along a common arc.

We shall employ the $p$-harmonic approximation
$$
   \begin{cases} \Delta_{p}u_p\,=\,0\qquad\text{in}\qquad G,\\ \phantom{ \Delta_{p}}
      u_p\,=\,0\qquad\text{on}\qquad \partial \Omega,\\\phantom {\Delta_{p}}
      u_p\,=\,1 \qquad\text{on}\qquad \partial K,
    \end{cases}
  $$
  for $\mathbf{p>2}$. It is known that $u_p\in C(\overline G)$ and it takes the correct values (in the classical sense) at each boundary point. We shall need the following results from \cite{L} (see also \cite{Ja}):
  \begin{enumerate}
  \item The level curves $\{u_p=c\}$ are convex, if $0\leq c\leq 1$,
  \item $u_p\nearrow u_\infty$ uniformly in $\overline G$,
  \item $|\nabla u_p|\neq 0$ in $G$,
  \item $u_p$ is real analytic in $G$,
  \item $\Delta u_p\leq 0$.
  \end{enumerate}
  The streamlines of $u_p$ do not meet in $G$. This is due to the regularity of $u_p$   and the Picard-Lindel\"of theorem. Properties 1), 3), and 5) are preserved at the limit $p=\infty$. Especially, $\nabla u_{\infty} \neq 0$ in $G$.

 We keep the \emph{normalization}  $\dist (\partial\Omega,  K) = 1$. Then $|\nabla u_\infty|\leq 1$, but we also need a uniform bound for $|\nabla u_p|$.
   The bound 
\begin{equation}
\label{eq:grad1}
|\nabla u_p|\leq 1 \quad \text{on }\partial \Omega.
\end{equation}
   follows by comparison with the distance function 
  $$
  \delta(x)=\dist(x,\partial \Omega).
  $$
  In a convex domain, $\delta$ is a supersolution of the $p$-Laplace equation. Since 
  $$
  0\leq u_p(x)\leq \delta(x)\quad \text{on } \partial G,
  $$
the same inequality also holds in $G$. In general, $|\nabla u_p|$ is unbounded (but $|\nabla u_\infty|\leq 1$), so we have to consider a subdomain, say $\{u_p < c\}$.
  
  \begin{lemma}\label{lem:gradbound} The uniform bound 
  \begin{equation}
  \label{eq:gradboundlem}
|\nabla u_p(x)|\leq \Bigl(\frac{1}{1-c}\Bigr)^\frac{1}{p-2}
  \end{equation}
  holds when $u_p(x)\leq c$, $0<c<1$.
  \end{lemma}
  
  \begin{proof}
  Let $\Upsilon_p(c)$ denote the level curve $\{u_p=c\}$ and 
  $$
  \delta_p(x) = \dist (x,\Upsilon_p(c)).
  $$ 
  Since $|\nabla u_p|$ obeys the maximum principle and $|\nabla u_p|\leq 1$ on $\partial \Omega$ by \eqref{eq:grad1}, it is enough to control $|\nabla u_p|$ on $\Upsilon_p(c)$. We see that 
  \begin{equation}
  \label{eq:grad2}
  c\leq u_p(x)\leq c+(1-c)\frac{\delta_p(x)}{\dist(\Upsilon_p(c), \partial K)}
  \end{equation}	
   on $\Upsilon_p(c)$ and on $\partial K$, i.e., on the boundary of $\{1 >u_p>c\}$. Again, the majorant is a supersolution to the $p$-Laplace equation, and hence \eqref{eq:grad2} holds in $\{1 >u_p>c\}$ by the comparison principle. It follows that
  \begin{equation}
  \label{eq:gradbound}
  |\nabla u_p(x)|\leq \frac{1-c}{\dist(\Upsilon_p(c),\partial K)}, 
  \end{equation}
  on\footnote{Since $u_p\nearrow u_\infty$, $\dist(\Upsilon_p(c),\partial K)$ increases with $p$. Thus we get an upper bound independent of $p$. This is sufficient for our purpose.} $\Upsilon_p(c)$.
  
  To get the explicit upper bound in \eqref{eq:gradboundlem}, we assume that $x_0\in \partial K$ is a point at which the distance $\dist(\Upsilon_p(c),\partial K)$ is attained. Let $R$ be the radius of the  largest ball $B_R(x_0)\subset \Omega$. Then 
  $$
  u_p(x)\geq 1-\left(\frac{|x-x_0|}{R}\right)^\frac{p-2}{p-1}\quad \text{in }B_R(x_0)\setminus K
  $$
  by comparison. Here the minorant is $p$-harmonic in $B_R(x_0)\setminus \{x_0\}$. Now 
  $$
    1-\left(\frac{|x-x_0|}{R}\right)^\frac{p-2}{p-1} = c \iff |x-x_0|=R(1-c)^{1+\frac{1}{p-2}} = r_c
  $$
  and clearly $\dist(\Upsilon_p(c),\partial K)\geq r_c$. We have by \eqref{eq:gradbound}
  $$
  |\nabla u_p(x)|\leq \frac{1}{R(1-c)^\frac{1}{p-2}}.
  $$
  To conclude, use $R\geq \dist(\partial\Omega, \partial K)=1$. 
  \end{proof}
  
\section{Equicontinuity of $|\nabla u_p|$}\label{sec:eqcont}

We shall prove that   
$$
\lim_{p\to \infty}{|\nabla u_p|} = |\nabla u_\infty|
$$
locally \emph{uniformly} in $G$. From \cite{KZZ} we can extract the following important properties: If $D\Subset G$, then 
\[\tag{\textbf{I}}
\iint_D |\nabla u_p-\nabla u_\infty|^2 \, dx_1 dx_2\to 0,\quad \text{as }p\to \infty,\]
\begin{equation}\tag{\textbf{J}}
\iint_D |\nabla (|\nabla u_p|^2)|^2 \, dx_1 dx_2\leq M_D<\infty,\end{equation}
for all (large) $p$.

The constant $M_D$ depends on $\|\nabla u_p\|_{L^\infty(E)}$, where $D\Subset E\Subset G$, and $\dist (D, \partial G)$, but not on $p$.

In \cite{KZZ} the estimates were derived  for solutions $u^\e$ of the auxiliary equation
$$
\Delta_\infty u^\e + \e \Delta u^\e = 0
$$
while we use $\Delta_pu_p = 0$ written as
$$
\Delta_\infty u_p +\frac{1}{p-2}\,|\nabla u_p|^2\Delta u_p = 0.
$$
The advantage of our approach is that the inequality $\Delta u_p\leq 0$ is available in convex domains for $p\geq 2$.

The conversion from $u^\e$ to $u_p$ requires only obvious changes. Formally, the factor $\e$ in front of an integral in \cite{KZZ} should be moved in under the integral sign and then replaced by $|\nabla u_p|^2/(p-2)$, upon which every $u^\e$ be replaced by $u_p$. This procedure is explained in our Appendix.

In order to prove that the family $\{|\nabla u_p|\}$ is locally equicontinuous, we shall use a device due to Lebesgue in \cite{Leb}. A function $f\in C(\overline B_R)\cap W^{1,2}(B_R)$ is monotone (in the sense of Lebesgue) if 
$$
\osc_{\partial B_r}f = \osc_{\overline B_r} f, \quad 0<r<R,
$$
where $B_r$ are concentric discs.  For such a function
\begin{equation}\label{eq:lebosc}
\Bigl(\osc_{B_r} f\Bigr)^2\ln \frac{R}{r}\leq \pi \iint_{B_R} |\nabla f|^2 \, dx_1 dx_2. 
\end{equation}
The proof is merely an integration in polar coordinates, cf. \cite{Leb}. We shall apply this oscillation lemma on the function $f = |\nabla u_p|^2$. It was shown by Bojarski and Iwaniec in \cite{BI} that the mapping 
 $$
  \frac{\partial u_p}{\partial x_1}-  i\,\frac{\partial u_p}{\partial x_2}, \quad i^2=-1,
  $$
  is quasiregular. That property implies that its norm  $|\nabla u_p|$ satisfies the maximum principle, and, where $|\nabla u_p|\neq 0$, also the minimum principle. Thus $|\nabla u_p|$ is monotone. So is $|\nabla u_p|^2$. From \eqref{eq:lebosc} we obtain 
  $$
  \Bigl(\osc_{B_r}\{|\nabla u_p|^2\}\Bigr)^2\ln \frac{R}{r}\leq \pi \iint_{B_R} |\nabla (|\nabla u_p|^2)|^2 \, dx_1 dx_2. 
  $$
  The uniform bound in (2) and a standard covering argument for compact sets yields the following result.
  
  \begin{thm} \label{thm:eqcont}(Equicontinuity) Let $D\Subset G$. Given $\e>0$, there is $\delta = \delta (\e,D)$ such that the inequality
  $$
\Big||\nabla u_p(x)|-|\nabla u_p(y)|\Big|<\e \quad \text{when $|x-y|<\delta$}, \quad x,y\in D,
  $$
  holds simultaneously for all $p > 2$.  
  \end{thm}
  Since $\nabla u _p\to \nabla u_\infty$ in $L^2_\text{loc}(G)$ we can use Ascoli's theorem to conclude that 
  $$
  \lim_{p\to \infty}|\nabla u_p| = |\nabla u_\infty|
  $$
  \emph{locally uniformly}. (More accurately, we have to extract a subsequence in Ascoli's theorem, but since the limit $|\nabla u_\infty|$ is unique, this precaution is not called for here.)
  
  \textbf{Caution:} The more demanding convergence $\nabla u_p\to \nabla u_\infty$ holds a.e., but perhaps \emph{not} locally uniformly.
  
  Let us finally mention that the \emph{uniform} convergence is not global. For example, in the ring $0<|x|<1$ we have 
  $$
  u_p(x)=1-|x|^\frac{p-2}{p-1}, \quad u_\infty = 1-|x|.
  $$
  Now $|\nabla u_p|$ is not even bounded near $x=0$. Thus the convergence cannot be uniform in the whole ring.
  
  \section{Convergence of the Streamlines}
  In this section, we study the convergence of the streamlines and prove Theorem \ref{thm:speed}.
 It is plain that the level curves $\{u_p=c\}$ converge to the level curves $\{u_\infty = c\}$.  However, the convergence of the streamlines requires a more sophisticated proof. (The problem is the identification of the limit as an $\infty$-streamline.)
  
  Suppose that we have the streamlines $\ba_p$ and $\ba_\infty$ having the same initial point $\ba_p(0)=\ba_\infty(0)=x_0$. Now
$$
  \frac{d\ba_p(t)}{dt} = \nabla u_p (\ba_p(t)), \quad  \frac{d\ba_\infty(t)}{dt} = \nabla u_\infty (\ba_\infty(t))
  $$
  when $0 < t < T_p$, where $u_p(\ba_p(T_p))= 1$. Thus
 $$
  \ba_p(t_2)-\ba_p(t_1) = \int_{t_1}^{t_2} \nabla u_p(\ba_p(t)) dt.
$$
Using the bound
$$
|\nabla u_p|\leq \Bigl(\frac{1}{1-c}\Bigr)^{\frac{1}{p-2}}	, \quad \text{when }u_p\leq c, 
$$
in Lemma \ref{lem:gradbound} we see that 
\begin{equation}
\label{eq:alpha3}
| \ba_p(t_2)-\ba_p(t_1)|\leq \Bigl(\frac{1}{1-c}\Bigr)^\frac{1}{p-2}|t_2-t_1|
\end{equation}
as long as the curves are below the level $u_p=c$, i.e., $u_p(\ba(t_2))\leq c$. In particular, the bound is valid in the domain $\{u_\infty<c\}$, where $c<1$. Thus, the family of curves is locally equicontinuous. By Ascoli's theorem we can extract a sequence $p_j\to \infty$ such that 
$$
\ba_{p_j}(t)\to \ba(t)
$$
uniformly in every domain $\{u_\infty<c\}$. Here $\ba(t)$ is some curve with initial point $\ba(0)=x_0$.

The endpoint of $\ba$ is on $\partial K$. Indeed, let $t_p=t_p(c)$ denote the parameter value at which $u_p(\ba_p(t_p))=c$. Take any convergent sequence, say $t_p\to t^*$. Then 
$$
c=\lim_{p\to\infty} u_p(\ba_p(t_p)) = u_\infty(\ba(t^*)).
$$
Thus $t^* = t_\infty(c)$. Then $t_p(c)\to t_\infty(c)$ for all $c$.

 By \eqref{eq:alpha3}
$$
| \ba(t_2)-\ba(t_1)|\leq |t_2-t_1|.
$$
Rademacher's theorem for Lipschitz continuous functions implies that $\ba(t)$ is differentiable at a.e. $t$. 

We claim that $\ba = \ba_\infty$. Since they start at the same point, the uniqueness of $\infty$-streamlines shows that it is enough to verify
$$
\frac{d\ba(t)}{dt} = \nabla u_\infty (\ba(t)).
$$
To this end, we shall employ the convex functions $F_p(t)=u_p(\ba_p(t))$. Indeed, 
$$
\frac{dF_p(t)}{dt} = \Big\langle \nabla u_p(\ba_p(t)), \frac{d\ba_p(t)}{dt}\Big\rangle = |\nabla u_p(\ba_p(t))|^2
$$
and 
$$
\frac{d^2F_p(t)}{dt^2} = 2\,\Delta_\infty u_p(\ba_p(t)) = -\frac{2}{p-2}\,\Delta u_p(\ba_p(t))\,|\nabla u_p(\ba_p(t))|^2.
$$
By Lewis's theorem, $\Delta u_p\leq 0$ in convex ring domains, if $p\geq 2$. Thus, 
$$
\frac{d^2F_p(t)}{dt^2}\geq 0
$$
and so \emph{the function $F_p(t)$ is convex}. The convergence
$$
F_{p}(t) = u_{p}(\ba_{p}(t))\to u_\infty(\ba(t)) = F(t)
$$
is at least locally uniform, when $p$ takes the values $p_1, p_2, p_3,\ldots$ extracted above. Also the limit $F(t)$ is convex, of course.

We have the locally uniform convergence
$$
 |\nabla u_{p}(\ba_p(t))|^2\to  |\nabla u_\infty(\ba(t))|^2,
$$
which follows from Theorem \ref{thm:eqcont} by writing 
$$
|\nabla u_p(\ba_p(t))|-|\nabla u_\infty(\ba(t))| = |\nabla u_p(\ba_p(t))|-|\nabla u_p(\ba(t))|+|\nabla u_p(\ba(t))|-|\nabla u_\infty(\ba(t))|.
$$
Thus, 
$$
\frac{dF_p(t)}{dt}=|\nabla u_{p}(\ba_p(t))|^2\,\to\, |\nabla u_\infty(\ba(t))|^2.
$$
It follows that\footnote{$\int |\nabla u_\infty(\ba(t))|^2\phi(t) dt\leftarrow  \int F_p'(t)\phi(t) dt = -\int  F_p(t)\phi'(t) dt \to -\int F(t) \phi'(t) dt$}  $F'(t) = |\nabla u_\infty(\ba(t))|^2$ for a.e. $t$. We also have by the chain rule
$$
\frac{dF(t)}{dt} = \Big\langle\nabla u_\infty(\ba(t)), \frac{d\ba}{dt}\Big\rangle
$$
a.e., since $\frac{d\ba}{dt}$ exists for a.e. $t$.

We have arrived at the identity
$$
|\nabla u_\infty(\ba(t))|^2 = \Big\langle\nabla u_\infty(\ba(t)), \frac{d\ba}{dt}\Big\rangle
$$
valid for a.e. $t$. From 
$$
\ba_p(t_2)-\ba_p(t_1) \leq \int_{t_1}^{t_2} |\nabla u_p(\ba_p(t)) |dt, 
$$
we get
$$
\ba(t_2)-\ba(t_1) \leq \int_{t_1}^{t_2} |\nabla u_\infty(\ba(t)) |dt, 
$$
  and, hence for a.e. $t$
  $$
  \Big\vert\frac{d\ba(t)}{dt}\Big\vert \leq |\nabla u_\infty(\ba(t)) |.
  $$
  We conclude that in the Cauchy-Schwarz inequality
  $$
  |\nabla u_\infty(\ba(t)) |^2 = \Big\langle\nabla u_\infty(\ba(t)), \frac{d\ba}{dt}\Big\rangle \leq |\nabla u_\infty(\ba(t))|\Big\vert \frac{d\ba}{dt}\Big\vert\leq |\nabla u_\infty(\ba(t)) |^2
  $$
  we have equality. It follows that 
  $$
 \frac{d\ba}{dt}= \nabla u_\infty(\ba(t))
  $$
  for a.e. $t$. In fact, it holds everywhere because now the identity
  $$
  \ba(t_2)-\ba(t_1) = \int_{t_1}^{t_2} \nabla u_\infty(\ba(t)) dt
  $$
  can be differentiated. This concludes our proof of the fact $\ba = \ba_\infty$.

  We see that the tangent $\frac{d\ba}{dt}$ is continuous. The proof reveals that the convex functions $F_p\to F$ uniformly and hence $F$ is convex as well. Therefore, its derivative 
  $$
  F'(t)= |\nabla u_\infty(\ba(t))|^2
  $$ is non-decreasing. In other words, $|\nabla u_\infty|^2$ is non-decreasing along the limit streamline.
  
 This proves Theorem \ref{thm:speed}.
  
  \section{Quadrilaterals and Triangles}
  Curved quadrilaterals and triangles, bounded by arcs of streamlines and level curves, are useful building blocks. It is tentatively understood that at least the interior of the figures are comprised in $G$; the level arcs  can be   on $\partial \Omega$ and, occasionally, on   $\partial K$.

  Recall that the $\infty$-streamline 
  $$
  \ba(t), \quad 0\leq t\leq T, 
  $$
  with initial point $\ba(0)=a\in \partial\Omega$ is unique and terminates at $\ba(T)$ on $\partial K$. On its way, it may (and usually does) meet other streamlines and has common parts with them. By Theorem \ref{thm:speed}, the speed 
  $$
  \left|\frac{d\ba(t)}{d t}\right| = |\nabla u_\infty (\ba(t))|
  $$
  is non-decreasing. Thus we have the bound\footnote{$$|\nabla u_{\infty}(\ba(T))| = \lim_{t\to T-}  |\nabla u_{\infty}(\ba(t))|$$} 
  $$
  |\nabla u_\infty (\ba(t_1))|\leq |\nabla u_\infty (\ba(t_2))|, \quad 0\leq t_1\leq t_2\leq T.
  $$
  Sometimes the result below (cf. Lemma 12 in \cite{LL}), valid for curved quadrilaterals and triangles, provides us with the reverse inequality, so that we may even conclude that the speed is constant along suitable arcs of streamlines.
\begin{figure}[h!]
  \begin{center}
 
	\begin{tikzpicture}[domain=-5:5,scale=1.2]
	       
\fill (3,0) circle[radius=2pt] node[below] {$b$};
\fill (-3,0) circle[radius=2pt] node[below] {$a$};
\fill (2,4) circle[radius=2pt] node[above] {$b'$};
\fill (-2,4) circle[radius=2pt] node[above] {$a'$};
\fill (0.5,3.25) circle[radius=2pt] node[above] {$\eta$};
\fill (0,-0.75) circle[radius=2pt] node[below] {$\xi$};
\draw (0,-0.75) .. controls (0.3,1) .. (0.5,3.25); \fill (0.3,1) node[left] {$\bm$};
\draw (-3,0) .. controls (0,-1) .. (3,0);\fill (-1.2,-0.8) node[below] {$\bs$};
\draw (-3,0) .. controls (-2,2) .. (-2,4);\fill (-2.2,2) node[left] {$\ba$};
\draw (3,0) .. controls (3,2) .. (2,4);\fill (3,2) node[right] {$\bbeta$};
\draw (2,4) .. controls (0.5,3) .. (-2,4); \fill (-0.5,3.6) node[above] {$\bo$};




	\end{tikzpicture}
	
\end{center}
\caption{The quadrilateral $abb'a'$.}
  \label{fig:quadbasic}  
	\end{figure}
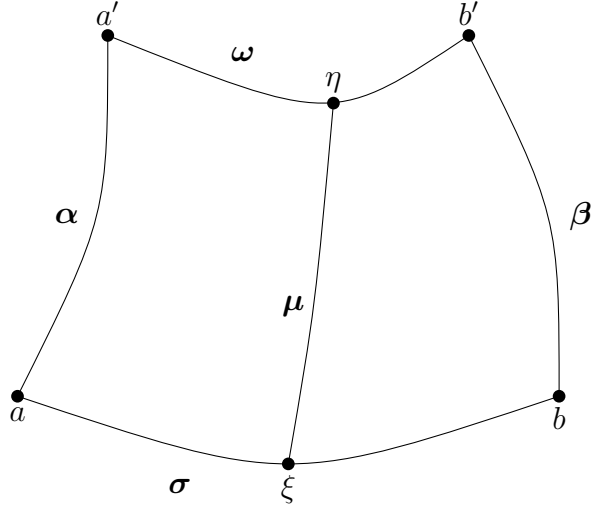
  \begin{lemma} \label{lem:ll} Suppose that the streamlines $\ba$ and $\bbeta$ together with the level curves $\bs$ (lower level) and $\bo$ (upper level) form a quadrilateral with vertices $a,b,b'$ and $a'$. If $\ba$ and $\bbeta$ do not meet before reaching $\omega$, then
  $$
  \max_{\overline{a'b'}} |\nabla u_\infty(\bo)|\leq  \max_{\overline{ab}} |\nabla u_\infty(\bs)|, 
  $$
  i.e., the maximal speed on the upper level is the smaller one.
  \end{lemma}
  
  Suppose now that $\xi\in \overline{ab}$ is a point on the lower level curve $\bs$ at which 
  $$
  |\nabla u_\infty(\xi)|= \max_{\overline{ab}} |\nabla u_\infty(\bs)| = M.
  $$
  Let $\bm$ be the streamline that passes through $\xi$. It intersects $\bo$ at some point $\eta\in \overline{a'b'}$ (it may have joined $\ba$ or $\bbeta$ before reaching $\eta$). See Figure \ref{fig:quadbasic}. The following result holds:
  
\begin{lemma}\label{lem:quad}  
We have
$$
|\nabla u_\infty(\bm)| = M \quad \text{on } \overline{\xi\eta}.
$$
Moreover, 
$$
\max_{\overline{a'b'}} |\nabla u_\infty(\bo)|= \max_{\overline{ab}} |\nabla u_\infty(\bs)|.
$$
  \end{lemma}
  \begin{proof} By Lemma \ref{lem:ll}
  $$
  |\nabla u_\infty(\xi)|\geq \max_{\overline{a'b'}} |\nabla u_\infty(\bo)|\geq |\nabla u_\infty(\eta)|
  $$
  and the monotonicity of the speed implies
  $$
 |\nabla u_\infty(\xi)|\leq |\nabla u_\infty(\bm(t))|\leq |\nabla u_\infty(\eta)|
  $$
  along the arc $\overline{\xi\eta}$ of $\bm$. Thus we have equality.
  \end{proof}
  
We can also formulate a similar result for curved triangles. Suppose that the streamlines $\ba$ and $\bbeta$ together with the level curve $\bs$ form a curved triangle with vertices $a,b$ and $c$. Assume again that $\xi\in \overline{ab}$ is a point at which 
  $$
  |\nabla u_\infty(\xi)|= \max_{\overline{ab}} |\nabla u_\infty(\bs)| = M.
  $$
  Let $\bm$ be the streamline that passes through $\xi$. It passes through $c$ (but may have joined $\ba$ or $\bbeta$ before reaching $c$). The following result holds:

\begin{cor}\label{cor:tri} For the triangle $a\,b\,c$
we have
$$
|\nabla u_\infty(\bm)| = M \quad \text{on } \overline{\xi c}.
$$
Moreover, 
$$
 |\nabla u_\infty(c)|= \underset{\overline{ab}}{\max} |\nabla u_\infty(\bs)|.
$$
\end{cor}  
\begin{proof} Take $\bo_i$ to be a sequence of level curves approaching $c$ from below. Then apply Lemma \ref{lem:quad} on the quadrilateral formed by $\bs, \bo_i, \ba$ and $\bbeta$ and let $i\to \infty$.
\end{proof}

\
  
  \paragraph{The Quadrilateral Rule.}
 We provide a practical rule for preventing meeting points. We keep the same notation.

\begin{prop}[Quadrilateral Rule]\label{prop:quad} If $|\nabla u(\bs(t))|$ is strictly monotone on the arcs $\overline{a\xi}$ and $\overline{\xi b}$ of the level curve $\bs$ (one of them may reduce to a point), then no streamlines can meet inside the quadrilateral. A streamline with initial point on the arc $\overline{ab}$ (but not $a$ or $b$) has constant speed $|\nabla u_\infty|$ till it meets $\ba, \bbeta$ or reaches $\bo$.
\end{prop}  

\begin{proof} Let $\bl= \bl(t)$ be a streamline passing through the point $x\in \overline{\xi b}$, $x\neq \xi$ on the level curve $\bs$.

\begin{figure}[h!]
  \begin{center}
 
	\begin{tikzpicture}[domain=-5:5,scale=0.6]
	       
\fill (3,0) circle[radius=2pt] node[below] {$b$};
\fill (-3,0) circle[radius=2pt] node[below] {$a$};
\fill (2,4) circle[radius=2pt] node[above] {$b'$};
\fill (-2,4) circle[radius=2pt] node[above] {$a'$};
\fill (0.5,3.25) circle[radius=2pt] node[above] {$\eta$};
\fill (0,-0.75) circle[radius=2pt] node[below] {$\xi$};
\draw (0,-0.75) .. controls (0.3,1) .. (0.5,3.25); \fill (0.3,1) node[left] {$\bm$};
\draw (-3,0) .. controls (0,-1) .. (3,0);\fill (-1.2,-0.8) node[below] {$\bs$};
\draw (-3,0) .. controls (-2,2) .. (-2,4);\fill (-2.2,2) node[left] {$\ba$};
\draw (3,0) .. controls (3,2) .. (2,4);\fill (3,2) node[right] {$\bbeta$};
\draw (2,4) .. controls (0.5,3) .. (-2,4); \fill (-0.5,3.6) node[above] {$\bo$};

\fill (1.5,-0.48) circle[radius=2pt] node[below] {$x$};
\draw[color=red, thick] (1.5,-0.48) to[out=85,in=-92] (0.35,1.6);
\fill (0.35,1.6) circle[radius=2pt] node[right] {$y$};



	\end{tikzpicture}
	
\end{center}
\caption{Case 1: impossible}
	\end{figure}
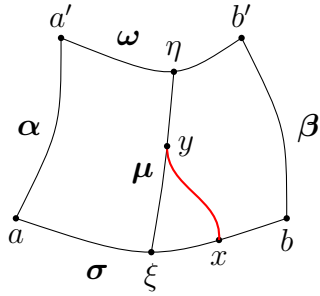

We have three cases: 1) If $\bl$ meets $\bm$ at the point $y$, then Lemma \ref{lem:quad}  applied on the quadrilateral $x b b'\eta y x$ (or Corollary \ref{cor:tri} if $\bm$ meets $\bbeta$, so that we have a triangle) implies
$$
M=|\nabla u_\infty(\bl)|
$$
on the whole arc $\overline{x\eta}$ of $\bl$ (or until $\bm$ reaches $\bbeta$). But then
$$
|\nabla u_\infty(\xi)| = |\nabla u_\infty(x)|,
$$
which contradicts the \emph{strict} monotonicity of $|\nabla u(\bs(t))|$.

\begin{figure}[h!]
  \begin{center}
 
	\begin{tikzpicture}[domain=-5:5,scale=0.6]
	       
\fill (3,0) circle[radius=2pt] node[below] {$b$};
\fill (-3,0) circle[radius=2pt] node[below] {$a$};
\fill (2,4) circle[radius=2pt] node[above] {$b'$};
\fill (-2,4) circle[radius=2pt] node[above] {$a'$};
\fill (0.5,3.25) circle[radius=2pt] node[above] {$\eta$};
\fill (0,-0.75) circle[radius=2pt] node[below] {$\xi$};
\draw (0,-0.75) .. controls (0.3,1) .. (0.5,3.25); \fill (0.3,1) node[left] {$\bm$};
\draw (-3,0) .. controls (0,-1) .. (3,0);\fill (-1.2,-0.8) node[below] {$\bs$};
\draw (-3,0) .. controls (-2,2) .. (-2,4);\fill (-2.2,2) node[left] {$\ba$};
\draw (3,0) .. controls (3,2) .. (2,4);\fill (3,2) node[right] {$\bbeta$};
\draw (2,4) .. controls (0.5,3) .. (-2,4); \fill (-0.5,3.6) node[above] {$\bo$};

\fill (1.5,-0.48) circle[radius=2pt] node[below] {$x$};


\draw[color=blue, thick] (1.5,-0.48) to[out=80,in=-72] (2.58,2.8);
\fill (2.58,2.8) circle[radius=2pt] node[right] {$y$};

	\end{tikzpicture}
	
\end{center}
\caption{Case 2: possible}
	\end{figure}
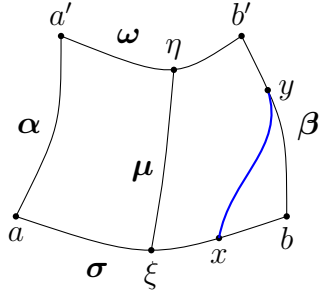

2) If $\bl$ meets $\bbeta$ at $y\in \overline{bb'}$, then Corollary \ref{cor:tri}  applied on the triangle $xby$ yields
$$
|\nabla u_\infty(\bl)| = \text{constant}
$$
on the arc $\overline{xy}$. 

\begin{figure}[h!]
  \begin{center}
 
	\begin{tikzpicture}[domain=-5:5,scale=0.6]
	       
\fill (3,0) circle[radius=2pt] node[below] {$b$};
\fill (-3,0) circle[radius=2pt] node[below] {$a$};
\fill (2,4) circle[radius=2pt] node[above] {$b'$};
\fill (-2,4) circle[radius=2pt] node[above] {$a'$};
\fill (0.5,3.25) circle[radius=2pt] node[above] {$\eta$};
\fill (0,-0.75) circle[radius=2pt] node[below] {$\xi$};
\draw (0,-0.75) .. controls (0.3,1) .. (0.5,3.25); \fill (0.3,1) node[left] {$\bm$};
\draw (-3,0) .. controls (0,-1) .. (3,0);\fill (-1.2,-0.8) node[below] {$\bs$};
\draw (-3,0) .. controls (-2,2) .. (-2,4);\fill (-2.2,2) node[left] {$\ba$};
\draw (3,0) .. controls (3,2) .. (2,4);\fill (3,2) node[right] {$\bbeta$};
\draw (2,4) .. controls (0.5,3) .. (-2,4); \fill (-0.5,3.6) node[above] {$\bo$};

\fill (1.5,-0.48) circle[radius=2pt] node[below] {$x$};

\draw[color=blue, thick] (1.5,-0.48) .. controls (1.8,2) .. (1.3,3.55); 
\fill (1.3,3.55) circle[radius=2pt] node[above] {$y$};


	\end{tikzpicture}
	
\end{center}
\caption{Case 3: possible}
	\end{figure}
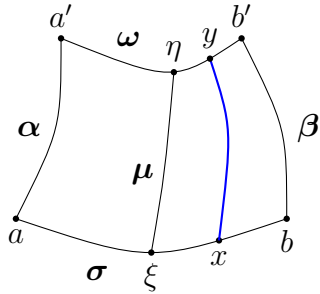

3) If $\bl$ passes through a point $y\in \overline{\eta b'}$ on the upper level $\bo$, $y\neq \eta$, $y\neq b'$, then Lemma \ref{lem:quad} applied on the quadrilateral $xbb'y$ (or Corollary \ref{cor:tri} in case of a curved triangle) yields
$$
|\nabla u_\infty(\bl)| = \text{constant}
$$
on the arc $\overline{xy}$.

Finally, if $x$ is chosen from the left level arc $\overline{a\xi}$, the proof consists of three similar cases again. Thus we have established that $\bl$ has constant speed till it first meets
$\ba, \bbeta$, or hits $\bo$.

It remains to show that no two streamlines can meet in the quadrilateral. 
A streamline $\bl$ passing through the point $x$ at the level curve $\bs$ has constant speed
$$
|\nabla u_{\infty}(x)| = |\nabla u_{\infty}(\bl)|
$$
till $\bl$ meets $\ba$, $\bbeta$ or hits $\bo$. But two meetings streamlines must have the same speed, which requires that they pass through $\bs$ at two points with the same speed $|\nabla u_{\infty}|$. By the \emph{strict} monotonicity of $|\nabla u_{\infty}(\bs)|$, this would require that the points are on different arcs $\overline{a\xi}$ and $\overline{\xi b}$. This is impossible, since no streamlines meet $\bm$. 
\end{proof}

The Quadrilateral Rule remains true if the monotonicity of $|\nabla u_{\infty}(\bs)|$ is not supposed to be strict. If $|\nabla u_{\infty}(\bs)|$ is constant on some subarc $\overline{cd}$, then the streamlines with initial points on $\overline{cd}$ are non-intersecting straight lines. To see this, we again consider the quadrilateral $a\,b\,b'\,a'$ bounded by $\ba,\bbeta,\bs,\bo.$

\begin{lemma}\label{lem:constant} 
  Assume  that $|\nabla u_\infty (\bs)|$ is constant on the arc $\overline{ab}$. Then no streamlines can meet inside the quadrilateral. Moreover,  $|\nabla u_\infty|$ is constant in the quadrilateral and all streamlines are straight lines.
\end{lemma}

\begin{proof} By Lemma \ref{lem:quad}, $|\nabla u_{\infty} (\bo)|$ is constant on the upper arc  $\overline{a'b'}$ .
  In particular, $|\nabla u_{\infty}|$ must be constant along $\ba$ and $\bbeta$. Then $|\nabla u_{\infty}|$ must be constant along any arc of a streamline passing through the quadrilateral. Every  point inside the quadrilateral lies on such a streamline. Therefore $|\nabla u_{\infty}|$ is constant in the quadrilateral, which means that it solves the \emph{Eikonal Equation}. Since $u_{\infty}$ is of class $C^1$, we can apply the next proposition
    to conclude that all streamlines are non-intersecting straight lines.
\end{proof}

\begin{prop} [Eikonal Equation]\label{prop:eikonal} Suppose that $v \in C^1(D)$ is a solution of the Eikonal Equation $|\nabla v| = C$ in the domain $D$, where $C$ denotes a constant. Then the streamlines of $v$ are non-intersecting segments of straight lines.
\end{prop}

\begin{proof} A very appeling direct proof is given in Lemma 1 in \cite{A2}.
\end{proof}

For the next result we abandon the \emph{strict} monotonicity in Proposition \ref{prop:quad}.

\begin{cor}[Quadrilateral Rule] \label{cor:constant} Assume 
  that $|\nabla u_\infty (\bs)|$ is monotone on the arc $\overline{ab}$. Then no streamlines can meet inside the quadrilateral. A streamline with initial point on the arc $\overline{ab}$ (but not $a$ or $b$) has constant speed till it meets $\ba$, $\bbeta$ or reaches $\bo$.
\end{cor}  

\begin{proof} Assume that $|\nabla u_{\infty} (\bs)|$ is non-decreasing. Consider the subarc $\overline{x^1x^2}$ on $\bs$ so that $|\nabla u_{\infty}(x^1)| \leq  |\nabla u_{\infty}(x^2)|$, where $x_1 < x_2$.  Let $\ba^j$ be the streamline passing through $x^j$. 
We claim that $\ba^1$ does not meet $\ba^2$ inside the quadrilateral. Indeed, suppose they meet at a point $c$ at the level line $\widetilde \bo$ before reaching $\bo$, where $\widetilde \bo$ intersects $\ba$ and $\bbeta$ at $a''$ and $b''$ respectively. Then Lemma \ref{lem:quad} applied to the quadrilaterals $a\,x^1\,c\,a''$ and  $a\,x^2\,c\,a''$  exhibit that the speeds
  $$|\nabla u_{\infty}(\ba^1(t))| = |\nabla u_{\infty}(\ba^2(t))| = |\nabla u_{\infty}(c)|$$
are constant along  the arcs. Again we see that the Eikonal Equation is valid in the triangle
  $x^1\,x^2\,c$. At the point $c$ this leads to a contradiction with Proposition \ref{prop:eikonal}. (Thus the eventual point $c$ must lie on $\bo$ and on $\partial K$.)
  
\end{proof}

    \paragraph{The Triangular Rule.} 
The above results may be formulated for a curved triangle as in Figure \ref{fig:tribasic} (seen as a degenerate quadrilateral).  Again,  suppose that the streamlines $\ba$ and $\bbeta$ together with the level curve $\bs$ form a curved triangle with vertices $a,b$ and $c$; $c$ is the meeting point of $\ba$ and $\bbeta$.  Assume that $\xi\in \overline{ab}$ is a point at which 
  $$
  |\nabla u_\infty(\xi)|= \max_{\overline{ab}} |\nabla u_\infty(\bs)| = M.
  $$
  Let $\bm$ be the streamline that passes through $\xi$. It passes through $c$ (but may have joined $\ba$ or $\bbeta$ before reaching $c$). 
 \begin{figure}[h!]
  
  \begin{center}
 
	\begin{tikzpicture}[domain=-5:5,scale=1.2]
	       
\fill (3,0) circle[radius=2pt] node[below] {$b$};
\fill (-3,0) circle[radius=2pt] node[below] {$a$};
\fill (0.5,3.25) circle[radius=2pt] node[right] {$c$};
\fill (0,-0.75) circle[radius=2pt] node[below] {$\xi$};
\draw (0,-0.75) .. controls (0.3,1) .. (0.5,3.25); \fill (0.3,1) node[left] {$\bm$};
\draw (-3,0) .. controls (0,-1) .. (3,0);\fill (-1.2,-0.8) node[below] {$\bs$};
\draw (-3,0) to [out=75, in =-95](0.5,3.25);\fill (-1.5,2) node[left] {$\ba$};
\draw (3,0) to [out=105, in =-95](0.5,3.25);\fill (2,2) node[right] {$\bbeta$};




	\end{tikzpicture}
	
\end{center}
\caption{The curved triangle $abc$.}
\label{fig:tribasic}  
	\end{figure}
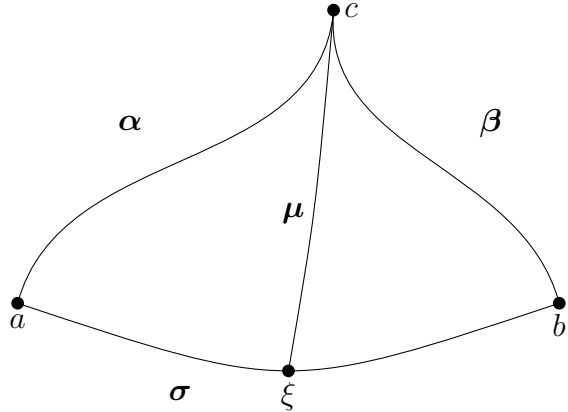
  By simply using the results for quadrilaterals, we may deduce the following.

\begin{cor}\label{cor:tricor} If $|\nabla u(\bs(t))|$ is strictly monotone on the arcs $\overline{a\xi}$ and $\overline{\xi b}$ of the level curve $\bs$ (one of them may reduce to a point), then no streamlines can meet inside the triangle. A streamline with initial point on the arc $\overline{ab}$ (but not $a$ or $b$) has constant speed $|\nabla u_\infty|$ till it meets $\ba$ or $ \bbeta$.
\end{cor}  

\begin{proof} If two streamlines meet at a point in the triangle we may construct a quadrilateral containing that point by letting $\bo$ be a level curve above $c$. Then Proposition \ref{prop:quad} yields a contradiction.
\end{proof}

\begin{lemma}\label{lem:trilem}   $|\nabla u_\infty (\bs)|$ cannot be  constant on a subarc of $\overline{ab}$, except if  $c \in \partial K.$
\end{lemma}

\begin{proof} We can again construct a triangle in which the Eikonal Equation is valid. This yields a contradiction, unless we allow a corner to be outside $G$.
\end{proof}

Vi can again abandon the \emph{strict} monotonicity. 

\begin{cor}[Triangular Rule] \label{cor:triconstant}
Suppose that $|\nabla u_\infty (\bs)|$ is monotone on the arc $\overline{ab}$ of the level curve $\bs$. Then no streamlines can meet inside the triangle. A streamline with initial point on the arc $\overline{ab}$ has constant speed till it meets $\ba$ or $\bbeta$.
\end{cor}

\begin{proof} Reason as in the proof of Corollary \ref{cor:tricor} and apply Corollary \ref{cor:constant}.
\end{proof}

  \section{Polygons}\label{sec:poly}
  Let $\Omega$ be a convex polygon with $N$ vertices $P_1,P_2,\ldots, P_N$ and set $P_{N+1}=P_1$. The gradient $\nabla u_\infty$ is continuous up to the boundary $\partial \Omega$ and especially at the vertices,
  $$|\nabla u_\infty(P_j)|=0, \quad j=1,2,\ldots,N.$$
  From each vertex $P_j$, there is a unique streamline $\bg_j$ that terminates on $K$. They are the attracting streamlines.

Let $M_j$ denote a point on the edge $\overline{P_jP_{j+1}}$ at which $|\nabla u_\infty|$ attains its maximum, i.e,
$$
|\nabla u_\infty(M_j)|=\max_{\overline{P_jP_{j+1}}} |\nabla u_\infty|. 
$$
The point divides the edge $\overline{P_jP_{j+1}}$ into two line segments $\overline{P_jM_j}$ and $\overline{M_jP_{j+1}}$. Denote by $\bm_j$ the streamline starting at the point $M_j$.

\begin{lemma}\label{lem:polymon} The normal derivative
$$
\frac{\partial u_\infty}{\partial n} = |\nabla u_\infty|
$$
is monotone along the half-edges $\overline{P_jM_j}$ and $\overline{M_jP_{j+1}}$ for $j=1,2,\ldots,N$.
\end{lemma}
\begin{proof}
 We arrange it so that the polygon is in the upper half-plane $x_2>0$ and the edge in question is on the $x_1$-axis, say the edge is
 $$a\leq x_1\leq b, \quad x_2=0.$$
 The convex level curves
 $$
 \{u_\infty=c\}
 $$
 approach the $x_1$-axis as $c\to 0$. The shortest distance from the level curve to the edge is attained at some point, say $(x_1(c), x_2(c))$. Choose a sequence $c_j\to 0$ so that $x_1(c_j)\to \xi$ and $x_2(c_j)\to 0$, where $(\xi,0)$ is some point, $a\leq\xi\leq b$ (in fact, $a < \xi < b$).  If $\xi >a$, let $a<\xi_1<\xi_2<\xi$ and keep $j$ so large that $\xi_2<x_1(c_j)$. The vertical lines $x_1=\xi_1$ and $x_1=\xi_2$ intersect the level curve $\{u_\infty = c\}$ at the points $(\xi_1,h_1^j)$ and $(\xi_2, h_2^j)$, i.e.
 $$
 u_\infty(\xi_1,h_1^j)= u_\infty(\xi_2,h_2^j)=c_j.
 $$
 The convexity of the level curve implies that $h_1^j\geq h_2^j$. (The chord between $(\xi_1,h_1^j)$ and $(x_1(c_j), x_2(c_j))$ must lie inside the set $\{u_\infty\geq c\}$.)
 It follows that the difference quotients in the normal direction satisfy
 $$
 \frac{u_\infty(\xi_1,h_1^j)-u_\infty(\xi_1,0)}{h_1^j}\,\leq \, \frac{u_\infty(\xi_2,h_2^j)-u_\infty(\xi_2,0)}{h_2^j},
 $$
 since both numerators are $= c_j-0$.
As $c_j\to 0$, also $h_1^j\to 0$ and $h_2^j\to 0$. By passing to the limit we obtain 
$$
|\nabla u_\infty(\xi_1,0)|\,\leq\, |\nabla u_\infty(\xi_2,0)|,\quad \xi_1<\xi_2<\xi
$$
as desired.

If $a<\xi<b$ we also obtain the reverse inequality for all $\xi<\xi_1<\xi_2< b$ so that we may conclude the desired result again. It also follows that $(\xi,0)$ is the $M_j$ point of this edge. This excludes that $\xi = a$ or $\xi = b$.
\end{proof}

  We are now ready to prove our main theorem for polygons. 
  
  \begin{proof}[Proof of Theorem \ref{thm:mainpoly}.] Consider the region bounded by $\overline{P_jP_{j+1}}, \bg_j, \bg_{j+1}$ and, if  $\bg_j$ does not meet $ \bg_{j+1}$ also $\partial K$. This can be either a curved triangle (meeting attracting streamlines) or a quadrilateral (the attracting streamlines do not meet). By Lemma \ref{lem:polymon}, $|\nabla u_\infty|$ is monotone along $\overline{P_jM_{j}}$ and $\overline{M_jP_{j+1}}$. Therefore, Corollary \ref{cor:constant} (in the case of a quadrilateral) and Corollary \ref{cor:triconstant} (in the case of a curved triangle) imply that no streamlines can meet (on either side of $\mu_j$)
    and that they have  constant speed until they meet $\bg_j$ or $\bg_{j+1}$, or hit $\partial K.$
  
  \end{proof}
  




  \section{General Domains}  \label{sec:general} In this section we assume that $\nabla u_\infty$ is continuous in $\overline \Omega\setminus K$  
   and that $|\nabla u_\infty|$ has a finite number of local minimum points and maximum points. Denote by $P_1, \ldots, P_N$ (with $P_{N+1}=P_1$ as before) the minimum points. From each $P_j$, there is a unique streamline $\bg_j$ that terminates in $K$. These streamlines divide $G$ into triangles with corners $P_k, P_k$ and $Q_k$ if $\bg_k$ and $\bg_{k+1}$ meet at $Q_k$, and quadrilateras with corners  $P_k, P_{k+1}, S_{k+1}$ and $S_{k}$ if $\bg_k$ and $\bg_{k+1}$ do not meet but they reach $K$ at the points $S_k$ and $S_{k+1}$. Recall the $\infty$-ridge, 
  $$
  \Gamma = \bigcup_{k=1}^N \{\bg_k(t), \quad 0\leq T\leq T_k\}.
  $$
  
   We give the proof of Theorem \ref{thm:general}.

  \begin{proof}[Proof of Theorem \ref{thm:general}.] Consider the region bounded by $\overline{P_jP_{j+1}}, \bg_j, \bg_{j+1}$ and perhaps $\partial K$. This can be either a curved triangle or  quadrilateral. By construction, $|\nabla u_\infty|$ is monotone along $\overline{P_jM_{j}}$ and $\overline{M_jP_{j+1}}$. Therefore, Corollary \ref{cor:constant} in the case of a quadrilateral and Corollary \ref{cor:triconstant} in the case of a curved triangle imply that no streamlines can meet (on either side of $\mu_j$) and that they are constant until they meet $\bg_j$ or $\bg_{j+1}$ or reach $\partial K$.
  \end{proof}

  \section{Appendix: Estimates of Derivatives of $|\nabla u_p|$}
  
  The fundamental properties
\[\tag{\textbf{I}}\label{eq:eqI}\iint_D |\nabla u_p-\nabla u_\infty|^2 \, dx_1 dx_2\to 0,\quad \text{as }p\to \infty,\]
\begin{equation}\tag{\textbf{J}}\label{eq:eqJ}\iint_D |\nabla (|\nabla u_p|^2)|^2 \, dx_1 dx_2\leq M_D<\infty,\end{equation}
for all (large) $p$
used in Section \ref{sec:eqcont} follow directly from \cite{KZZ}, where the corresponding estimates are ingeniously derived for the solution $u^\e$ of
$$
\Delta_\infty u^\e +\e\Delta u^\e  = 0.
$$
To transcribe the work to the solution $u_p$ of the $p$-Laplace equation 
$$
  \Delta_\infty u_p +\frac{1}{p-2}|\nabla u_p|^2\Delta u_p = 0
  $$
  one has to replace the constant factor $\e$ by the \emph{function} $|\nabla u_p|^2/(p-2)$ \emph{under} the integral sign. Below we give just a synopsis of the procedure, referring to the numbering of formulas and theorems in \cite{KZZ}. (The reader is supposed to have access to \cite{KZZ}.)
  
  Formula (2.5) in \cite{KZZ} becomes
  $$
  -\det(D^2 u_p) = |\nabla |\nabla u_p||^2+\frac{1}{p-2}(\Delta u_p)^2.
  $$
  Formula (2.7) becomes
  $$
  I_p(\phi)=\iint_U |\nabla |\nabla u_p||^2 \phi\, dx_1 dx_2 +\frac{1}{p-2}\iint_U (\Delta u_p)^2\phi \,dx_1 dx_2
  $$
  and (2.8)
  $$
  I_p(\phi)=\frac12\iint_U \Bigl(\Delta u_p\langle\nabla u_p, \nabla \phi\rangle- \sum_{i,j=1}^2\frac{\partial^2 u_p}{\partial x_i\partial x_j}\frac{\partial u_p}{\partial x_j}\frac{\partial \phi}{\partial x_i}  \Bigr)  \, dx_1 dx_2 .
  $$
  Lemma 5.1 is needed only for $\alpha = 2$ (and since $|\nabla u_p|\neq 0$ we can put $\kappa = 0$ in the proof). It becomes
\[
\begin{split}
&    \iint_U |\nabla |\nabla u_p|^2|^2\xi^2 \, dx_1 dx_2 +\frac{1}{p-2}\iint_U |\nabla u_p|^2(\Delta u_p)^2\xi^2 \,dx_1 dx_2\\
&\leq C(2)\iint_U|\nabla u_p|^4\left(|\nabla \xi|^2+|\xi||D^2\xi|\right) \, dx_1 dx_2.
\end{split}
\]
This yields Lemma 2.6 and the desired property \eqref{eq:eqJ}, since $|\nabla u_p|$ is locally bounded by Lemma \ref{lem:gradbound}. 

Lemma 5.2 is valid with no changes (replace $u^\e$ with $u_p$), but the proof uses Lemma 5.1 as above. Then Lemma 5.2 implies the flatness estimate in Lemma 2.7:
\[
\begin{split}
&\fiint_{B_r(x)}\left(|\nabla u_p|^2-\langle \nabla P, \nabla u_p\rangle\right)^2 \, dx_1 dx_2\leq C\left(\fiint_{B_{2r}(x)} |\nabla u_p|^4\, dx_1 dx_2\right)^\frac12 \\
&\times \left(\fiint_{B_{2r}(x)}\left(\frac{|u_p-P|^2}{r^2}\left(|\nabla P|+|\nabla u_p|\right)^2+\frac{|u_p-P|^4}{r^4}\right) dx_1 dx_2 \right)^\frac12
\end{split}
\]
valid for any linear function $P$. This estimate is needed for the proof of Theorem 1.4, when one has to identify the limit of $|\nabla u_p|^2$ in $L^2_\text{loc}$ as $|\nabla u_\infty|^2$. Theorem 1.4 contains our desired property \eqref{eq:eqI}.

 \bigskip
\paragraph{Acknowledgments:} Erik Lindgren was supported by the Swedish Research Council, 2017-03736. Peter Lindqvist was supported by The Norwegian Research Council, grant no. 250070 (WaNP).
\bigskip

\noindent {\textsf{Erik Lindgren\\  Department of Mathematics\\ Uppsala University\\ Box 480\\
751 06 Uppsala, Sweden}  \\
\textsf{e-mail}: erik.lindgren@math.uu.se\\

\noindent \textsf{Peter Lindqvist\\ Department of
   Mathematical Sciences\\ Norwegian University of Science and
  Technology\\ N--7491, Trondheim, Norway}\\
\textsf{e-mail}: peter.lindqvist@ntnu.no

\end{document}